\documentclass[a4paper,12pt]{amsart}
\usepackage{amsfonts}
\usepackage{mathrsfs}
\usepackage{amsmath,amssymb,latexsym,amsfonts,amscd}
\usepackage[small,nohug,UglyObsolete]{diagrams}
\diagramstyle[labelstyle=\scriptstyle]

\newcommand{\bC}{{\mathbb C}}

\newcommand{\bP}{{\mathbb P}}

\newcommand{\bR}{{\mathbb R}}

\newcommand\dep{\text{\rm dep}}

\newtheorem{thm}{Theorem}[section]
\newtheorem{lem}[thm]{Lemma}

\newtheorem{prop}[thm]{Proposition}
\newtheorem{claim}[thm]{Claim}

\theoremstyle{definition}
\newtheorem{defn}[thm]{Definition}

\newtheorem{rem}[thm]{Remark}
\theoremstyle{remark}

\title{Factoring $3$-fold flips and  divisorial contractions to curves}
\author{Jungkai A. Chen and Christopher D. Hacon }
\address{Department of Mathematics, National Taiwan University, Taipei,
106, Taiwan} \email{jkchen@math.ntu.edu.tw}

\address{Department of Mathematics, University of Utah, Salt Lake
  City, 155 South 1400 East, JWB 233, UT 84112-0090, USA}
\email{hacon@math.utah.edu}

\thanks{The first author was partially supported by TIMS, NCTS/TPE
and  National Science Council of Taiwan. The second author was
supported by NSF grant
  0757897.
  We are indebted to Hayakawa, Koll\'ar and  Mori for many useful discussion especially to Mori who  read our
  preliminary version and gave various comments and corrections.
Some of this work was done during a visit of the first author to
the University of Utah. The first author would like to thank the
University of Utah for its hospitality.}

\begin{document}
\begin{abstract}
We show that $3$-fold terminal flips and divisorial contractions to a curve
may be factored by a sequence of weighted blow-ups, flops, blow-downs to a locally complete intersection curve in a smooth $3$-fold or divisorial contractions to a point.
\end{abstract}
\maketitle
\pagestyle{myheadings} \markboth{\hfill J. A. Chen and C. D. Hacon \hfill}{\hfill Factoring $3$-fold flips and divisorial contractions to curves \hfill}
\section{introduction}
Flips, flops and divisorial contractions are the ``elementary'' birational
maps of the minimal model program.
Divisorial contractions are the higher dimensional
analog of (the inverse of) blowing up a smooth point on a surface.
They are morphisms contracting a divisor covered by $K_X$-negative curves.
Flips
are instead a new operation which consists of a ``surgery'' in
codimension $\geq 2$ which replaces certain $K_X$-negative curves by some
$K_X$-positive curves.
Flops are also surgeries in codimension $\geq 2$
but they replace certain $K_X$-trivial curves by some new
$K_X$-trivial curves.

These operations appear as follows.
Given a smooth projective variety of general type over $\mathbb C$,
then a minimal model
of $X$ may be constructed via a finite sequence of flips and divisorial contractions (cf. \cite{BCHM}). Minimal models are not unique, but is known that any two
minimal models are connected by a finite sequence of flops (cf. \cite{Kaw08}).

Not surprisingly, the geometry of divisorial contractions is better
understood than that of flips and flops. In fact the proof of the existence
of flips and flops (cf. \cite{BCHM}, \cite{HM})
is very abstract and gives little insight to their geometry.

In dimension $3$ these operations are reasonably well understood.
Koll\'ar has shown \cite{Kol89} that
if $X\dasharrow X^+$ is a flop (of terminal $3$-folds), then
$X$ and $X^+$ have analytically isomorphic singularities.
In his seminal paper \cite{Mo88}, Mori gives a precise
understanding of the geometry of all flips with irreducible flipping curve (extremal neighborhoods).

It should also be noted that $3$-fold divisorial
contractions may be subdivided into two classes:
those that map the divisor to a curve and those that map the
divisor to a point. Divisorial contractions to a point are classified in \cite{Co00},
\cite{Kaw96}, \cite{Kk01}, \cite{Kk02}, \cite{Kk03}, \cite{Kk05} and \cite{Luo98}.
There are also several partial results on divisorial contractions to a curve
cf. \cite{Cu}, \cite{Tz1},
\cite{Tz2}, \cite{Tz3}, \cite{Tz4} and \cite{Tz5}.

The purpose of this paper is to show that in dimension $3$,
we can factor flips and divisorial contractions to curves via the simpler operations given by flops, blow-downs to LCI curves (i.e. $C\subset Y$
a local complete intersection curve in a smooth variety) and divisorial contractions to points. More precisely we prove

\begin{thm} \label{divSig}
Let  $g: X \to W$ be a flipping contraction and
$\phi: X {\dashrightarrow} X^+$ be the corresponding flip, then
$\phi$ can be factored as
$$ X=X_0 \stackrel{f_0}{\dashrightarrow} X_1 \dashrightarrow\ldots
\dashrightarrow X_n \stackrel{f_n}{\dashrightarrow} X^+,$$
such that each $f_i$ is the inverse of a $w$-morphism, a flop, a blow-down to a LCI curve or  a divisorial contraction to a point.

Let $g: X \to W$ be a divisorial contraction to a curve, then $g$ can be factored as
$$ X=X_0 \stackrel{f_0}{\dashrightarrow} X_1 {\dashrightarrow}\ldots  {\dashrightarrow}X_n\stackrel{f_n}{\dashrightarrow} W,$$
such that each $f_i$ is the inverse of a $w$-morphism, a flop, a blow-down to a LCI curve or a divisorial contraction to a point.
\end{thm}
We remark that the proof of the above theorem follows by the classification results of \cite{Mo88} and the results of \cite{HaI} and \cite{HaII}.
One should also note that this kind of approach is not entirely new. See for example \cite{Kaw94} and \cite{Ha03}.

As a final motivation for \eqref{divSig}, we would like to mention a possible
application of these ideas to characteristic $p>0$.
In this context, the existence of divisorial contractions (in the big case)
is known by work of Keel \cite{Keel99}. The existence of flips is instead
only known in some special cases \cite{Kaw94} and \cite{Kol91}.
Theorem \ref{divSig} suggests that it may be possible to construct
flips in  characteristic $p>0$ via a sequence of divisorial
contractions (and extractions) and flops (which are hopefully easier
to construct than flips).

\section{Preliminaries}
We work over the field of complex numbers $\mathbb C$.
{\em Numerical} (resp. {\em linear}, $\mathbb Q$-{\em linear}) {\em equivalence} of divisors is denoted by $\equiv$ (resp. $\sim$, $\sim_{\mathbb Q }$),
if $f:X\to Y$ is a projective morphism, then $\equiv _Y$ denotes $f$-{\em numerical}
(also known as numerical over $Y$) {\em equivalence} and $\sim_{\mathbb Y }$ (resp. $\sim_{\mathbb Q , Y}$) denote {\em linear equivalence} (resp. $\mathbb Q$-{\em linear equivalence}) {\em over $Y$}. So, for example, $D\sim_{\mathbb Q , Y}D'$ if and only if there
are rational numbers $q_i\in \mathbb Q$, rational functions $f_i\in \mathbb C (X)$ and a $\mathbb Q$-Cartier divisor $L$ on $Y$ such that
$D-D'=\sum q_i(f_i)+f^*L$.
If $X$ is a normal variety, then $K_X$ denotes a canonical divisor on $X$.
A {\em log pair} $(X,B)$ is given by a normal variety $X$ and a
$\mathbb Q$-divisor $B$ such that $K_X+B$ is $\mathbb Q$-Cartier.
A {\em log resolution} of a log pair $(X,B)$ is a proper birational morphism
$\nu :X'\to X$
such that ${\rm Exc }(\nu )\cup \nu ^{-1}_*B$ is a simple normal crossings
divisor.
If $(X,B)$ is a log pair and $\nu :X'\to X$ is a  log resolution of $(X,B)$,
then we write $K_{X'}+\nu ^{-1}_*B=\nu ^* (K_X+B)+\sum a_iE_i$ where
$\nu _* K_{X'}=K_X$. The rational numbers $a_{E_i}(X,B):=a_i$ are the
{\em discrepancies} of $(X,B)$.
The log pair $(X,B)$ is {\em  terminal} (resp. {\em  canonical} or {\em plt}) if
$a_{E}(X,B)>0$ (resp. $a_{E}(X,B)\geq 0$ or $a_{E}(X,B)>-1$) for all exceptional divisors $E$ over $X$. The log pair $(X,B=\sum b_iB_i)$ is {\em klt}
if it is plt and $b_i<1$ for all $i$.
By inversion of adjunction, it is known that if $(X,S+B)$ is a log pair with
$\lfloor S+B\rfloor =S$, then $(X,S+B)$ is plt (resp. canonical or terminal)
if and only if
$(S,B_S)$ is klt (resp. canonical or terminal) where $(K_X+S+B)|_S=K_S+B_S$.

\subsection{Terminal threefold singularities}
Here we briefly recall several facts about terminal singularities in dimension $3$.
It is well known that terminal surfaces are smooth and hence terminal
threefolds have isolated singularities.
Let $P \in X$ be the germ of a terminal singularity, then $X$ has rational
singularities. The {\em index} of $P \in X$ is the smallest integer $r>0$
such that $rK_X$ is Cartier on a neighborhood of $P$.
If $P\in X$ is a point of index $r$ and $D$ is any Weil divisor on $X$, then $rD$ is also Cartier on a neighborhood of $P$.
The canonical cover $\pi: (P^\sharp \in X^\sharp ) \to (P \in X)$ is a cyclic cover of degree $r$ such
that $P^\sharp$ is an isolated $cDV$ point of index $1$ (cf.
\cite[3.1]{Reid83}). The analytic germ $(P^\sharp \in X^\sharp)$ can be embedded
into $\bC^4$ as a hypersurface $\varphi=0$. The singularity $(P \in
X)$ can thus be realized as a cyclic group $\mu_r$-quotient of a
hypersurface singularity.  Note that there exists a semi-invariant
coordinate system $y_1,\ldots, y_4$ near $(P^\sharp \in X^\sharp)$.
These singularities are classified in \cite{Mo85}
(see also \cite[\S 1a]{Mo88} for a brief survey).

The weights of the coordinates
$y_1,\ldots, y_4$ are usually denoted by
$\frac{1}{r}(a_1,a_2,a_3,a_4)$ \cite[pg. 243]{Mo88}.
Two weights $\frac{1}{r}(a_1,a_2,a_3,a_4)$ and
$\frac{1}{r}(b_1,b_2,b_3,b_4)$ are said to be equivalent if
there is an integer $\lambda$ relatively prime to $r$, such that $b_i
\equiv a_i \lambda \ \ (\text{mod} \ r)$ for all $i$.
Recall that for any terminal singularity $P\in X$ there is a deformation producing $k\geq 1$ terminal cyclic quotient singularities $P_1,\ldots , P_k$.
The number $k=aw(P\in X)$ is called the {\em axial weight} of $(P\in X)$.
We let $$aw(X)=\sum _{P\in {\rm Sing}(X)}aw(P\in X).$$
Note that the axial weight coincides with the {\em axial multiplicity} defined in \cite[Definition-Corollary 1a.5.iii]{Mo88} except in the case $cAx/4$.
We may assume that the singularities $P_i$ have type $\frac 1 {r_i}(1,-1,b_i)$ where
$0<b_i\leq r/2$.
The collection $\{ P_1,\ldots , P_k \}$ is known as the
 {\em basket of singularities } of $(X,P)$ and it can be written as
$$\mathcal B(P\in X)=\{n_i\times (b_i,r_i)|i\in I,\ n_i\in \mathbb Z ^+ \},$$
where $n_i$ denotes the number of times that a point $P_i$
representing a singularity $\frac 1 {r_i}(1,-1,b_i)$ appears.

To a given variety $X$, we can associate the basket of singularities
$\mathcal B=\cup _{P\in {\rm Sing}(X)}\mathcal B(P\in X)$.
We can define the following invariants corresponding to a singularity
$$\sigma (P\in X)=\sum _{i=1}^{aw(P\in X)}b(P_i),\qquad \sigma (X)
=\sum _{P\in {\rm Sing}(X)}\sigma (P\in X),$$
$$\Xi (P\in X)=\sum _{i=1}^{aw(P\in X)}r(P_i),\qquad \Xi (X)
=\sum _{P\in {\rm Sing}(X)}\Xi (P\in X).$$

\begin{rem}\label{r-basket} The basket $\mathcal B(P\in X)$ is uniquely determined by the singularity
$(P\in X)$. We have
$$\begin{array}{c|ccccc} \label{table1}
\text{type} & \text{G. E.} & aw &  \text{basket}
 & \sigma & \Xi \\
\hline
cA/r &  A_{kr-1} & k  & k \times (b,r) & kb & kr\\
cAx/2 & D_{k+2} & 2  & 2 \times (1,2) & 2 & 4\\
cAx/4 &  D_{2k+1} & k & \{ (1,4), (k-1) \times
(1,2)\} & k & 2k+2\\
cD/2 & D_{2k} & k & k \times (1,2) & k & 2k \\
cD/3 & E_6 & 2 & 2 \times (1,3) & 2 & 6 \\
cE/2 & E_7 & 3  & 3 \times (1,2) & 2 & 6
\end{array}$$
where the column G. E. lists the singularities for general elephant (i.e. a general element) $E_X\in |-K_X|$.
\end{rem}

\subsection{Extremal neighborhoods}\label{ss-en}
Recall that (cf. \cite[\S 1]{KM92}) an {\em extremal neighborhood}
$f:X\supset C\to Y\ni Q$ is a proper bimeromorphic morphism of
complex $3$ dimensional spaces $f:X\to Y$ such that
\begin{enumerate} \item $X$ is terminal;
\item $Y$ is normal with distinguished point $Q\in Y$;
\item $f^{-1}(Q)$ consists of a single irreducible curve $C\subset X$;
\item $K_X\cdot C<0$.\end{enumerate}
If instead of the condition $K_X\cdot C<0$, we have the condition $K_X\cdot C=0$, then we will say that $f$ is a {\em $K$-trivial extremal neighborhood}.
An {\em isolated extremal neighborhood} is an extremal neighborhood $f:X\supset C\to Y\ni Q$ such that $\dim {\rm Exc}(f)=1$. In this case
${\rm Exc}(f)=C$ and $K_Y$ is not $\mathbb Q$-Cartier.

An extremal neighborhood is {\em divisorial} if  $\dim {\rm Exc}(f)=2$.
In this case, ${\rm Exc}(f)$ is an irreducible divisor, $K_Y$ is
$\mathbb Q$-Cartier and $Y$ is terminal.

By \cite[pg. 151, 2.3.2]{Mo88}, \begin{equation}\label{e-2}K_X
\cdot C= -1+ \sum_{P \in C} w_P(0),\end{equation}  where
$$w_P(0)=\left\{ \begin{array}{ll} 0 & \text{ if } P \text{ is of
index } 1, \\ 0\leq \frac {c}{r}
\le \frac{r-1}{r} & \text{ if } P \text{ is of index }
r>1.
\\ \end{array} \right. $$
The exact value of $w_P(0)$ is given in \cite[pg. 175, 4.9]{Mo88}.

Extremal neighborhoods are classified in \cite{Mo88}, see also \cite{KM92} and \cite{Mo07}.
We have the following fundamental result also known as Reid's general elephant conjecture (cf. \cite{Mo88}).
\begin{thm}\label{t-ge} Let $f:X\supset C\to Y\ni Q$  be an extremal neighborhood and
$E_X\in |-K_X|$ and $E_Y\in |-K_Y|$ be general elements. Then $E_X$ and $E_Y$ are normal and have only Du Val singularities.
\end{thm}

\subsection{Extremal contractions}
A {\em flipping contraction} (resp. {\em flopping contraction}) is a projective birational morphism $f:X\to Y$
from a terminal $3$-fold to a normal variety such that $\dim {\rm Exc}(f)=1$,
$\rho (X/Y)=1$,
$f_*\mathcal O _X =\mathcal O _Y$ and $-K_X$ is $f$-ample (resp. $K_X$ is $f$-trivial).
A {\em flipping curve} (resp. {\em flopping curve}) is a curve $C\subset {\rm Exc}(f)$.
The {\em flip } (resp. {\em flop}) of a  flipping contraction (resp. a {flopping contraction}) $f:X\to Y$
is a birational morphism $f^+:X^+\to Y$ where $X^+$ is a terminal $3$-fold,
 $\dim {\rm Exc}(f^+)=1$,
$\rho (X^+/Y)=1$,
$f^+_*\mathcal O _{X^+} =\mathcal O _Y$ and $K_{X^+}$ is $f$-ample (resp. $f$-trivial).
A {\em flipped curve} (resp. a {\em flopped curve}) is a curve $C\subset {\rm Exc}(f^+)$.
One defines the flip (resp. flop) of an extremal neighborhood (resp. $K$-trivial extremal neighborhood) analogously.

A {\em divisorial contraction} is a birational morphism $f:X\to Y$
from a terminal $3$-fold to a normal variety such that $\dim {\rm Exc}(f)=2$,
$\rho (X/Y)=1$,
$f_*\mathcal O _X =\mathcal O_Y$ and $-K_X$ is $f$-ample.
It is known that the contracted divisor $F\subset X$ is irreducible and $Y$ is terminal.

We will need the following results.
\begin{thm}\label{t-flip} Let $f:X\supset C\to Y\ni P$ be a flipping (resp. flopping contraction). Then the flip (resp. flop) $f^+:X^+\to Y$ exists.
Moreover, in the analytic category, $\varphi :X\dasharrow X^+$ factors via a finite
sequence of flips (resp. flops) $\varphi _i :X_i\dasharrow X_i^+=X_{i+1}$ with
irreducible flipping (resp. flopping) curves
$f_i:X_i\to Y_i$ over $Y$ such that
${\rm Exc }(\varphi _i)= C_i\cong \mathbb P ^1$ and ${\rm Exc }(\varphi _i^+)= C_i^+\cong \mathbb P ^1$.
\end{thm}
\begin{proof} See \cite{Mo88} and \cite{Ka}.
\end{proof}
\begin{rem}\label{r-div} Notice that the arguments of \cite{Ka} also imply that
if $f:X\supset C\to Y\ni P$ is a contraction of a divisor to a curve and
$C=\cup C_j$ is not irreducible, then (in the analytic category)
$f$ factors through an extremal neighborhood.
\end{rem}
\begin{rem}\label{t-el} It is expected that if $X$ is a terminal $3$-fold and $f:X\to Y$ is the germ of a flipping or divisorial contraction to a curve and
if $E_X\in |-K_X|$ is general, then the divisor $E_X$ is normal with Du Val singularities.
\end{rem}

\begin{thm}\label{t-Cu} Let $X$ be a terminal Gorenstein
$3$-fold, $f:X\to Y$ an
extremal contraction. Then $f$ is divisorial and if it contracts the divisor
to a curve $C$, then $X$ has only $A_n$ singularities,
$C\subset Y$ is LCI (i.e. $Y$ is smooth near $C$ and $C$ is a local complete
intersection in $Y$) and $X$ is the blow up of $Y$ along $C$.
If $f$ contracts the divisor
to a point $P\in Y$, then $f:X\to Y$ is classified (it belongs to one of the
four possibilities listed in  \cite{Cu}).
\end{thm}

\begin{proof}
See \cite{Cu}.
\end{proof}

We will need the following easy lemmas.
\begin{lem}\label{l-e2} Let $P\in X$ be a terminal $3$-fold singularity of
index $r>1$ and $f:Y\supset F\to X\ni P$ be an extremal divisorial contraction
of discrepancy $1/r$.
Let $G_X\in |-K_X|$ be a normal divisor with only Du Val singularities and $G_Y$ its strict transform on $Y$, then \begin{enumerate}
\item $P\in {\rm Supp}(G_X )$, \item $f^*(K_X+G_X)=K_Y+G_Y$,
\item $f|_{G_Y}:G_Y\to G_X$ is not an isomorphism, and \item $f|_{G_Y}$ is dominated by the
minimal resolution $\tilde G_X\to G_X$.\end{enumerate}
\end{lem}
\begin{proof} We have that $f^*(K_X+G_X)=K_Y+G_Y+tF$. Since $-K_X$ is not Cartier at $P$, we have that $P\in {\rm Supp}(G_X)$.
It follows that $t\geq 0$. Since $G_X$ is canonical, by inversion
of adjunction, $(X,G_X)$ is canonical. Therefore, $t\leq 0$ and hence
$t=0$. Since $F$ is a $\mathbb Q$-divisor, so is $G_Y\cap F\ne \emptyset$.
It follows that $f|_{G_Y}$ is not an isomorphism.
Since $G_X$ is Du Val and
$K_{G_Y}=(f|_{G_Y})^*K_{G_X}$, $f|_{G_Y}$ is dominated by the
minimal resolution $\tilde G_X\to G_X$.
\end{proof}
\begin{lem}\label{l-e1} Let $X$ be a terminal analytic
$3$-fold and $G_X\in |-K_X|$ be a normal divisor with only
Du Val singularities. If $E_X\in |-K_X|$ is general, then $E_X$ is a normal
divisor with only Du Val singularities and $\rho (\tilde E_X/E_X)
\leq \rho(\tilde G_X /G_X)$
where $\tilde E_X$ and $\tilde G_Y$ denote the minimal resolutions of $E_X$ and $G_X$.
\end{lem}
\begin{proof} Since $G_X$ has Du Val singularities, then it is canonical and hence $(X,G_X)$ is canonical. But then $(X,E_X)$ is canonical cf. \cite[\S4]{Kol95} and so $E_X$ is Du Val. The remaining assertion follows easily.
\end{proof}

\subsection{Divisorial extractions}
 \begin{thm}\label{Ha} Let $P \in X$ be a point of index $r>1$.
\begin{enumerate}
\item There is a weighted blow up $Y \to X$ at $P \in X$ with
discrepancy ${1}/{r}$.

\item Extremal divisorial contractions to $P \in X$ with
discrepancy ${1}/{r}$ are given by the weighted blow-ups classified in \cite{HaI}, \cite{HaII}.
\item Extremal divisorial contractions to $P \in X$ with
discrepancy grater than ${1}/{r}$ are classified in \cite{Kk05}.
\item There is  a partial resolution
$$X_n \to X_{n-1} \to \ldots \to X_0=X,$$ such that $X_n$ has Gorenstein singularities and each map is
a weighted blowup over a singular point of index $r_i>1$ as in
(1). We have that $n\leq \rho (\tilde E_X/E_X)$ where $E_X\in |-K_X|$ is general and
$\tilde E_X\to E_X$ is the minimal resolution.
\end{enumerate}
\end{thm}
\begin{proof} See \cite{Kaw96}, \cite{HaI} and \cite{HaII}.\end{proof}
\begin{defn}
In the sequel, a $w$-{\em morphism} will denote an extremal divisorial contraction to a point $P \in X$ of index $r>1$ with
discrepancy ${1}/{r}$ see \eqref{Ha}. A partial resolution as in (4) of \eqref{Ha} will be called a $w$-{\em resolution} of $P\in X$.
We define $\dep (X,P)$ the {\em depth} of $X$ at a point $P\in X$ to be
the minimum length of any $w$-resolution of (a neighborhood) of $P\in X$.
\end{defn}
\begin{rem} We have $\dep(X,P)=0$ if and only if $X$ is Gorenstein at $P$.
We have that the sum of $\dep(X,P)$ for $P\in X$ is bounded below
by the difficulty of $X$ (an invariant used to show termination of
flips) however, the inequality can be strict cf.
\cite[2.11]{HaIII}.
\end{rem}
\begin{lem}\label{l-cyclic} If $P\in X$ is a cyclic point of index $r$, then $\dep (X)=r-1$.
\end{lem}
\begin{proof} The only $w$-morphisms are given by weighted blow ups $f:Y\to X$ such that $Y$ contains (at most) two quotient singularities of indices $a$ and $r-a$. The proof now follows easily by induction on the index.
\end{proof}
\begin{prop}\label{p-cA} If $(P\in X)$ is of
$cA/r$ type, then $\Xi(P\in X)-aw(P\in X)\leq \dep (P\in X)\leq \Xi(P\in X)-1$.
\end{prop}
\begin{proof} We follow \cite{HaI}. The singularity $P\in X$ is given by the equation $\varphi=xy+g(z^r,u)=0$ with weights $\frac 1 r (\beta , -\beta , 1, r)$.
Write $g(z^r,u)=\sum a_{i,j}z^{ri}u^j$, then we let $$\lambda:=aw(P\in X)={\rm min}\{j|a_{0,j}\ne 0\},$$ and for any $s>0$, we let
$$\nu _s:={\rm min}\{si+j|a_{i,j}\ne 0\}.$$
Notice that $\nu _1\leq \nu _2\leq \cdots $ and $\nu _s=\lambda $ for any $s\geq \lambda$.  Let $$t:={\rm min }\{ s>0|\nu _s=\lambda \},$$ then clearly $1\leq t\leq \lambda$.
We will show that
$$\dep(P\in X)= \lambda r-t.$$
The only $w$-morphisms $f:X_1\to X$ are weighted blow ups with weights $\frac 1 r (r_1,r_2,1,r)$ such that $r_1+r_2=r\nu _1$ \cite[\S 6]{HaI}.
We may assume that  $f:X_1\to X$ is the first $w$-morphism in a $w$-resolution of $P\in X$ computing $\dep (X)$.
One computes that on $X_1$ there are two quotient singularities $Q_1$ of type
$\frac 1 {r_1}(r,-1,-r)$ and $Q_2$   of type
$\frac 1 {r_2}(r,-1,-r)$ and (possibly) a singularity $Q_4$ of type $cA/r$ defined by $$xy+g(z^ru,u)/u^{\nu _1}=xy +\sum a_{i,j}z^{ri}u^{i+j-\nu _1}=0$$ with weights $\frac 1 r (\beta , -\beta, 1, r)$.
If $\lambda =\nu _1$ then $Q_4\not\in X_1$ and
$\dep (P\in X)=r_1+r_2-1=\nu _1r-1$.
Otherwise $$\dep (P\in X)=\dep (X_1)+1=\nu_1r-1+\dep (Q_4\in X).$$
Notice that
$$\left\{ \begin{array}{l}
\lambda(Q_4 \in X_1)=aw(Q_4\in X_1):=\min\{j-\nu_1| a_{0,j}\ne 0 \} = \lambda -\nu _1 \\
\nu_{s-1}(Q_4 \in X_1) =\min\{(s-1)i+(i+j-\nu_1) |a_{i,j} \ne 0\}  =\nu_{s} -\nu _1 \end{array} \right.$$ So $t(Q_4 \in X_1)=t-1$.  By induction on $aw$,
we have $\dep (Q_4\in X_1)=(\lambda-\nu _1) r-(t-1)$ and hence
$\dep(P\in X)= \lambda r-t$ the assertion now follows as
$\lambda r=aw(P\in X)r= \Xi (P\in X).$
\end{proof}
\begin{prop}\label{p-de} We have the following
\begin{enumerate}
\item if $(P\in X)$ is of type $cAx/4$, then $\dep (P\in X)\leq \Xi(P\in X)-1$,
\item if $(P\in X)$ is of type $cD/2$ or $cD/3$ then $\dep (P\in X)\leq \Xi(P\in X)$, and
\item if $(P\in X)$ is of type $cE/2$, then $\dep (P\in X)\leq \Xi(P\in X)+1$.\end{enumerate}
\end{prop}
\begin{proof} This follows immediately from (4) of \eqref{Ha} and \eqref{r-basket}.
\end{proof}
\begin{prop}\label{p-dep} Let $f:Y\supset E\to X\ni P$ be (the germ of) a divisorial contraction to a point. Then $$\dep (Y)\geq \dep(X)-1.$$
\end{prop}
\begin{proof} We follow the classification of \cite{Kk05}.
There are several cases to consider.

If $P\in X$ is Gorenstein, then $\dep (P\in X)=0$ and the claim is clear.

If $f$ is an $w$-morphism, then by definition of $\dep(X)$, we have
$\dep (X)\leq \dep (Y)+1$ as required.

If $ P\in X$ is of type $cA/r$, then $f$ is a weighted blow up
of discrepancy $a/r$ (cf. \cite[Thm. 1.2.i]{Kk05}).
With the notation of Proposition \eqref{p-cA}, then this is obtained by using
a weight $\frac 1 r (r_1,r_2,a,r)$ such that $r_1+r_2\equiv 0$ $({\rm mod}\ ar)$, $$\frac{r_1+r_2}{r}=\tau-wt(g)=\nu_a$$
where
$\tau-wt(z,u)=(\frac{a}{n},1)$, and
$z^{\frac{r_1+r_2}{a}}=z^{\frac{r\nu_a}{a}} \in g$.
It follows that $\nu_s \le \frac{s \nu_a}{a} < \nu_a$ for all $s <a$ and
therefore $t \ge a$.

There are two quotient singularities $Q_1,Q_2\in Y$ of index $r_1$ and $r_2$ and possibly a $cA/r$ point $Q_4\in Y$ defined by
$$\{xy-g(z^ru^a,u)/u^{\nu _a}=0\}/\mathbb Z _r(-r_1, -r_2 , -a , r)$$
where $g(z^ru^a,u) /u^{\nu _a}=\sum a_{i,j}z^{ri}u^{ai+j-\nu _a}$.
It follows that  $$\sum \dep(Q_i\in Y)=r_1-1+r_2-1+\dep (Q_4\in Y)=r\nu _a -2+ \dep (Q_4\in Y).$$
Proceeding as in \eqref{p-cA}, one sees that
$\dep (Q_4\in Y)=(\lambda -\nu _a)r-(t-a)$
and hence $$\dep (Y)=\lambda r -t+a-2 =\dep (X) +a-2 \ge \dep(X)-1.$$

The remaining cases are classified in \cite[1.2, 1.3]{Kk05}.
If $P \in X$ is not of the type $cA/n$ and the
discrepancy is $a/n>1/n$, then  there are $6$ cases.
Two cases of ordinary type with arbitrarily large
discrepancies
which are given by weighted blow-ups (cf. \cite[Theorem
1.2.ii]{Kk05}), and  $4$ cases of exceptional type (cf. \cite[Table
3]{Kk05}). We would like to remark that in general, there may be some {\it
hidden non-Gorenstein point} on $Y$ in (i.e.  non-Gorenstein points on $Y$
at which the exceptional divisor is Cartier cf. \cite[pg. 59]{Kk05}).
However, such hidden non-Gorenstein points do not occur in the
cases of exceptional type (e1, e2, e11) because we have $a/n=4/2$ or $2/2$ and thus $2K_Y \sim_X 4E$ or $2K_Y \sim_X 2E$.
It follows that if $E$ is Cartier at a point $Q \in Y$ then $K_Y$ is numerically equivalent to a Cartier divisor and so $K_Y$ is Cartier at $Q \in Y$. Hence, $Y$ is Gorenstein at $Q$.
Therefore the baskets of $Y$ and $X$ are completely described by \cite[1.3]{Kk05}.

Recall that by Reid's singular Riemann-Roch formula we have that
$$\chi(2K_X)=\frac{1}{2}K_X^3-3\chi(\mathcal O_X)+\sum_{\mathcal B (X)} \frac{b_i(r_i-b_i)}{2r_i},$$
and therefore
$$\chi(2K_Y)-\chi(2K_X)=\frac{1}{2}(\frac a n E)^3+\sum_{\mathcal B (Y)} \frac{b_i(r_i-b_i)}{2r_i}-\sum_{\mathcal B( X)} \frac{b_i(r_i-b_i)}{2r_i}.$$
\begin{claim}\label{coh-0} Let $f:Y\supset E\to X\ni P$ be a divisorial contraction of terminal $3$-folds such that $a(E,X)=1$ (resp. $a(E,X)=2$). Then
$$\chi (\mathcal O _Y(2K_Y))\geq \chi (\mathcal O _X(2K_X))+1.$$ 
\end{claim}
\begin{proof} If $K_Y=f^*K_X+E$, then $R^kf_*\mathcal O _Y(2K_Y-E)=0$ for $k>0$
by Kawamata-Viehweg vanishing.
Since $Y$ is terminal it is Cohen-Macaulay and hence so is $E$ cf. \cite[5.25]{K-M}.
Therefore $K_E=(K_Y+E)|_E$ and we have a short exact sequence (cf. \cite[5.26]{K-M})
$$0\to \mathcal O _Y(2K_Y-E)\to \mathcal O _Y(2K_Y)\to \mathcal O_E(K_E)\to 0.$$
Since $-K_E=-2E|_E$ is ample, $h^1(\mathcal O_E(K_E))=0$ and so
$R^1f_*  \mathcal O _Y(2K_Y)=0$. Moreover, we have $R^2f_*  \mathcal O _Y(2K_Y)
\cong R^2f_*  \mathcal O _E(K_E)\cong H^2(\mathcal O _E(K_E))\cong \mathbb C$.
By a Leray spectral sequence computation, we obtain
$$\chi(\mathcal O_Y(2K_Y))=\sum (-1)^i\chi(R^if_* \mathcal O_Y(2K_Y))= \chi(f_* \mathcal O_Y(2K_Y))+\chi(R^2f_* \mathcal O_Y(2K_Y))$$
$$=\chi(f_* \mathcal O_Y(2K_Y))+h^0(R^2f_* \mathcal O_Y(2K_Y))=\chi (\mathcal O _X(2K_X))+1.$$

If $K_Y=f^*K_X+2E$, then  by Kawamata-Viehweg vanishing, $R^kf_*\mathcal O _Y(2K_Y-2E)=0$ for $k>0$.
By Serre duality we have $$h^1(\mathcal O _E(2K_Y))=h^1(\mathcal O _E(4E))=h^1(\mathcal O _E(K_E-4E)).$$ By Kawamata-Viehweg vanishing, one sees that $h^1(\mathcal O _E(K_E-4E))=0$ and
hence $h^1(\mathcal O _E(2K_Y))=0$.

Similarly $h^1(\mathcal O _E(2K_Y-E))=0$.
From the short exact sequences
$$0\to \mathcal O_Y(2K_Y-(l+1)E)\to \mathcal O_Y(2K_Y-lE)\to \mathcal O_E(2K_Y-lE)\to 0$$
for $l\in \{1,0 \}$, one sees that $R^1f_*\mathcal O_Y(2K_Y-E)=R^1f_*\mathcal O_Y(2K_Y)=0$ and $\mathbb C \cong R^2f_*\mathcal O _Y(2K_Y-E)\hookrightarrow R^2f_*\mathcal O _Y(2K_Y)$.
The required inequality follows similarly to the previous case.
\end{proof}

\noindent{\bf e1.} $f$ is of type $e1$ with
$\frac{a}{n}=\frac{4}{2}$. We have that $\frac{4}{2}E^3=\frac{4}{r}$ by Table 2 of \cite{Kk05} (compare with Table 3, were  $r=2r'$).   Notice also that there is only a
quotient singularity $Q \in Y$ of type $\frac{1}{2r'}(1,-1,r'-4)$.
Thus the basket of $Y$ is $(r'-4,2r')$ and
we  have $\dep(Y) = 2r'-1$.

On the other hand $P \in X$ is a point of type $cD/2$. Notice
that $\sigma(P \in X)=aw(P \in X)$ in this case and $\mathcal B (X)
=\{aw(P \in X) \times (1,2)\}$.

We now estimate $\dep (X)$.
By the singular Riemann-Roch formula, we can compute:
$$ \chi(2K_Y)-\chi(2K_X)=\frac{1}{2} 8E^3+ \frac{(r'-4)(r'+4)}{4r'} - \frac{1}{4} aw(P \in X).$$
Since $\chi(2K_Y)-\chi(2K_X) \ge 1$ cf. \eqref{coh-0} and $E^3=2/r$, we
have
${r'}-1 \ge aw(p \in X).$ Therefore, by \eqref{p-de}
$$ \dep(X) \le \Xi(X)=2 aw(P\in X) \le 2 r'-2=\dep(Y)-1$$ as required.

$f$ is of type $e1$ with
$\frac{a}{n}=\frac{2}{2}$. We have that
$\frac{2}{2}E^3=\frac{4}{r}$ with $r=2r'$, $P \in X$ is a point
of the type $cD/2$ and $Q\in Y$ is a quotient singularity
of type $\frac{1}{2r'}(1,-1,r'-2)$.
We thus have $\dep(Y) = 2r'-1$.

By the singular Riemann-Roch
formula, we have
$$ \chi(2K_Y)-\chi(2K_X)=\frac{1}{2} E^3+ \frac{(r'-2)(r'+2)}{4r'} - \frac{1}{4} aw(P \in X).$$
Since $\chi(2K_Y)-\chi(2K_X) \ge 1$ cf. \eqref{coh-0}, we
have that ${r'}-1 \ge aw(P \in X)$ so that
$$ \dep(X) \le \Xi(X)=2 aw(P\in X) \le 2 r'-2=\dep (Y)-1$$ as required.

\noindent{\bf e2.} $f$ is of type $e2$ with
$\frac{a}{n}=\frac{2}{2}$.  We have that
$\frac{2}{2}E^3=\frac{2}{r}$ with $r=2r'$ and $P \in X$ is a point
of the type $cD/2$. There is a singularity $Q \in Y$ of type
$cA/2r'$ deforming to  $2 \times \frac{1}{2r'}(1,-1,r'-1)$.
We thus have $\dep(Y) \in \{ 4r'-1,4r'-2 \}$.


By the
singular Riemann-Roch formula, we compute:
$$ \chi(2K_Y)-\chi(2K_X)=\frac{1}{2} E^3+ 2 \frac{(r'-1)(r'+1)}{4r'} - \frac{1}{4} aw(P \in X).$$
Since $\chi(2K_Y)-\chi(2K_X) \ge 1$ cf. \eqref{coh-0}, we thus
have $ {2r'}-1 \ge aw(P \in X)$. Hence
$$ \dep(X) \le \Xi(X)=2 aw(P\in X) \le 4 r'-2\le \dep (Y) $$ and we are done.

\noindent{\bf e11.} If $f $ is of type $e11$, then $P\in X$ is of
type $cE/2$ and so $\dep (X)\leq 7$. There are two quotient
singularities $Q_1,Q_2\in Y$ of type $\frac{1}{2}(1,1,1)$ and
$\frac{1}{6}(1,-1,-1)$ respectively. Thus $\dep(Y)=6\geq \dep (X)-1$
as required.

\noindent{\bf o3}. The remaining case to consider is classified in \cite[1.2.ii]{Kk05}. $f$ is a divisorial contraction to a point $P\in X$ of type $cD/2$ with discrepancy $a/2 > 1/2$. $f$ is described explicitly as a weighted blow up and the required inequality follows from a direct computation which we describe below.

We first consider the case in \cite[1.2.ii.a]{Kk05}, so we assume that the embedding is given by $\varphi =0$ where
$$\begin{array}{ll}
\varphi &=u^2+y^2z+g(x^2,z)+uxq(x^2,z)+\lambda yx^{2\alpha-1} \\
 & =u^2+y^2z+\sum a_{ij} x^{2i}z^j+u\sum b_{ij} x^{2i+1}z^j+\lambda yx^{2\alpha-1}
 \end{array}$$
and $\mathbb Z _2$ acts with weights $\frac{1}{2}(1,1,0,1)$.

The corresponding divisorial contraction is given by a weighted blowup $Y \to X$ with  weights
$\sigma_a=\frac{1}{2}(a,r,2,r+2)$. We have that $a|(r+1), a\equiv r \equiv 1 \quad (\text{mod} \ 2)$, and so we may
assume that $r+1=2ad$.  We may also assume that $a \ge 3$.
Since the weight of $\varphi$ with respect to $\sigma _a$ is $r+1$, we
have
$$\left\{  \begin{array}{l}\min\{ai+j| a_{ij} \ne 0 \} \ge r+1=2ad,\\
 \min\{ (2i+1)a+2j | b_{i,j}\ne 0\} \ge r=2ad-1, \\
 (2 \alpha-1)a \ge r+2=2ad+1. \end{array} \right.\eqno(\ddagger)$$

One sees that $Q_2$ (resp. $Q_4$) is a cyclic quotient point of index $r$ (resp. $r+2$).
$Q_3$ is a  point of index $2$ given by
$$
\begin{array}{ll}\bar \varphi &=u^2z+y^2+\bar{g}(x^2,z)+ux\bar{q}(x^2,z)+\lambda
yx^{2\alpha-1}z^{\bar{\delta}}\\
&=u^2z+y^2+\sum a_{ij} x^{2i}z^{\bar{\beta}_{ij}}+u\sum b_{ij}
x^{2i+1}z^{\bar{\gamma}_{ij}}+\lambda
yx^{2\alpha-1}z^{\bar{\delta}}
\end{array}\eqno(\dagger)$$

where $$\left\{ \begin{array}{ll}\bar{\beta}_{ij}&=ai+j-r-1, \\
\bar{\gamma}_{ij}&=\frac{2ai+a+2j-r}{2},\\
\bar{\delta}&=\frac{2a\alpha-a-r-2}{2}. \end{array} \right.$$

{\bf Claim.} $Q_3$ is a terminal singularity of type $cD/2$.

To see this, we use the coordinates
$\bar u:=y+\frac{\lambda}{2}x^{2 \alpha-1}z^{\bar{\delta}}$ and
$\bar y:=u+\frac{1}{2}\sum_{\bar{\gamma}_{ij} \ne 0} b_{ij}
x^{2i+1}z^{\bar{\gamma}_{ij}-1}$. Then we have
$$\bar\varphi =\bar u^2+\bar y^2z+ \bar y(\sum_{\bar{\gamma}_{ij} =0} b_{ij}
x^{2i+1}) +\bar g(x^2,z) .$$ for some $\bar g$. We may assume that
$(\sum_{\bar{\gamma}_{ij} =0} b_{ij}
x^{2i+1})=\bar\lambda\bar x^\epsilon$ for some $\epsilon>0$ and  $\bar\lambda
\in \mathbb C$ by replacing $x$ by a
suitable coordinates $\bar x$. Thus, this is a
terminal singularity of type $cD/2$. $\qed$

Notice that $Q_1$ is possibly a singular
point of index $a$. However, since all hidden non-Gorenstein points are of index
$2$, and $a \ge 3$ is an odd
integer, $Q_1$ is not a singular point of $Y$.
Therefore, there
is a term of the form $x^b$ in $\varphi$ with $wt(x^b)=r+1$. In
other words, $x^{4d} \in \varphi$.
In conclusion, we have that $$\dep(Y)=(r-1)+(r+1)+\dep(Q_3).$$

We will now compute $\dep(X)$ by comparing with $\dep(Y)$.
We will
construct a sequence of weighted blow ups
$$Z_a\to \ldots \to Z_{k+1}\to Z_{k}\to \ldots \to Z_0=X$$ such that
\begin{enumerate}
\item each map $\phi_k: Z_{k+1} \to Z_k$ is a weighted blowup with weights $\sigma =\sigma_{\rm ev}:=\frac{1}{2}(1,2d-1,2,2d+1)$ if $k$ is even and $\sigma =\sigma_{\rm odd}:=\frac{1}{2}(1,2d+1,2,2d-1)$ if $k$ is odd;

\item moreover, let $U_k^i$ be the standard affine piece of the weighted blowup $Z_k \to Z_{k-1}$ and $R_k^i$ be the origin of $U_k^i$. The weighted blowup $\phi_k: Z_{k+1} \to Z_k$ center at $R_k^3$.
\end{enumerate}

We need the following claims.

\noindent{\bf Claim 1.} On $Z_k$, there are two cyclic singularities $R^+_k,R^-_k$ of index $2d+1$, $2d-1$ respectively such that $\{R^+_k,R^-_k\}=\{R^2_k,R^4_k\}$. \\
\noindent{\bf Claim 2.} On $Z_k$, there is a $cD/2$ singularity $R^3_k$ and $R^3_a \in U^3_a \cong U_3 \ni Q_3$. \\
\noindent{\bf Claim 3.} $Z_k$ is terminal and $Z_k-\{ R^2_k,R^3_k,R^4_k\}$ is Gorenstein for all $k$.\\
\noindent{\bf Claim 4.} $\phi_k: Z_{k+1} \to Z_k$ is a $w$-morphism for all $k$.\\

Grant these claims for the time being. We then have that
$$\dep(X) \le a+\dep(R^3_a)+\sum_{k=1}^a (\dep(R^+_k)+\dep(R^-_k))$$
$$=a+\dep(Q_3)+a(2d-2+2d)=a+\dep(Q_3)+4ad-2a$$
$$=\dep(Q_3)+2r+2-a=\dep(Y)+2-a.$$
Hence $\dep(Y) \ge \dep(X)+a-2 \ge \dep(X)+1$ as required.

\begin{proof}[Proof of Claims 1, 2, 3 and 4.]
One can check that the equation $\varphi_k$ of $Z_k \cap
U^3_k $ has $\sigma-wt(\varphi_k)=2d$ and can be written as

$$ \begin{array}{ll} u^2z+y^2+\sum a_{ij}
x^{2i}z^{{\beta}_{ij}(k)}+u\sum b_{ij}
x^{2i+1}z^{{\gamma}_{ij}(k)}+\lambda
yx^{2\alpha-1}z^{{\delta}(k)}, & \text{ if $k$ is odd}\\
u^2+y^2z+\sum a_{ij} x^{2i}z^{{\beta}_{ij}(k)}+u\sum b_{ij}
x^{2i+1}z^{{\gamma}_{ij}(k)}+\lambda
yx^{2\alpha-1}z^{{\delta}(k)}, & \text{ if $k$ is even.}
\end{array}
$$
Where $\beta_{ij}(0)=j$ and $\beta_{ij}(k+1)=\beta_{ij}(k)+i-2d $
so that $$\beta_{ij}(k)=k(i-2d)+j;$$
$ \gamma_{ij}(0)=j$ and
$\gamma_{ij}(k)=\gamma_{ij}(k-1)+\frac{2d+1}{2}+\frac{2i+1}{2}-2d$ if $k$
is odd and $\gamma_{ij}(k)=\gamma_{ij}(k-1)+\frac{2d-1}{2}+\frac{2i+1}{2}-2d$
if $k$ is even,
so that $$\begin{array}{ll} \gamma_{ij}(k)=k \frac{2i+1}{2}-kd+j, &\text{ if $k$ is even} \\
                          \gamma_{ij}(k)=k \frac{2i+1}{2}-kd+j+\frac{1}{2}, &\text{ if $k$ is odd;} \end{array}$$
$\delta(0)=0$ and
$\delta(k)=\delta(k-1)+\frac{2d-1}{2}+\frac{2\alpha-1}{2}-2d$ if $k$ is odd
and $\delta(k)=\delta(k-1)+\frac{2d+1}{2}+\frac{2\alpha-1}{2}-2d$
if $k$ is even, so that
$$\begin{array}{ll} \delta(k)=k \frac{2\alpha-1}{2}-kd, &\text{ if $k$ is even} \\
                          \delta(k)=k \frac{2\alpha-1}{2}-kd-\frac{1}{2}, &\text{ if $k$ is odd.} \end{array}$$

We need to verify that $\sigma -wt(\varphi_k)=2d$ with respect to the weight $\sigma_{\rm odd}=\frac{1}{2}(1,2d+1, 2, 2d-1)$ or $\sigma_{\rm ev}=\frac{1}{2}(1,2d-1,2,2d+1)$ (depending on the parity of $k$).
One can easily see  that $\sigma_{\rm odd}-wt(u^2z)=2d$ and $\sigma_{\rm odd}-wt(y^2) = 2d+1$
(resp. $\sigma_{\rm ev}-wt(u^2)=2d+1$ and $\sigma_{\rm ev}-wt(y^2z) = 2d$ ) if $k$ is odd (resp. even). Moreover, since $x^{4d} \in \varphi$, we have $x^{4d} \in \varphi_k$ for all $k$.
Therefore, $wt(\varphi_k) \le 2d$. One sees that if $\beta_{ij}(k+1) \ge 0$, $\gamma_{ij}(k+1) \ge 0$ and $\delta(k+1) \ge 0$, then $wt(\varphi_k) \ge 2d$.

Therefore it is enough to  show that
$\beta_{ij}(k) \ge 0$, $\gamma_{ij}(k) \ge 0$ and $\delta(k) \ge 0$ for all $1\leq k \leq a$ and all $i$ and $j$ with $a_{i,j}\ne 0$, resp. $b_{i,j}\ne 0$.

Recall that by $(\ddagger )$, $ai+j \ge 2ad$, thus $$\beta_{ij}(k)=k(i-2d)+j=\frac{k}{a}(ai+j)-\frac{kj}{a}-2kd+j \ge j -\frac{kj}{a}=j \frac{a-k}{a} \ge 0.$$
Next, recall that by $(\ddagger )$, $(2i+1)a+2j \ge 2ad-1$, hence $$\gamma_{ij}(k) \ge k \frac{2i+1}{2}-kd+j =\frac{k}{2a}(2i+1)a-kd+j $$
$$\ge \frac{k}{2a}(2ad-1-2j) -kd+j=\frac{2aj-2kj-k}{2a} \ge \frac{-k}{2a} \ge \frac{-1}{2}.$$
Also, recall that  by $(\ddagger )$, $(2\alpha-1)a \ge 2ad+1$, hence $$\delta(k) \ge k \frac{2\alpha-1}{2}-kd-\frac{1}{2} =\frac{k}{2a}(2\alpha-1)a-kd-\frac{1}{2}$$
 $$\ge \frac{k}{2a}(2ad+1) -kd-\frac{1}{2}=\frac{k}{2a}-\frac{1}{2} \ge \frac{-1}{2}.$$
 The above inequalities now follow because $\beta_{ij}(k)$ and $\delta(k)$ are integers.

Note that $\beta_{ij}(a)=\bar{\beta}_{ij},
\gamma_{ij}(a)=\bar{\gamma}_{ij}$ and $\delta(a)=\bar{\delta}$.
Thus $R^3_a \in Z_a$ is indeed isomorphic to $Q_3 \in Y$.

Claims 1 and 2 now follow easily from the equations $\varphi _k$.

Claim 3 also follows by explicit computation.

We now verify that the homogeneous part $\varphi_{k,\sigma -wt=2d}$,
which defines the exceptional divisor  $E_k$ in $\bP(1,2d+1,2,2d-1)$
(resp. $\bP(1,2d-1,2,2d+1)$),  is irreducible. To see this, notice
that as we have seen $x^{4d} \in \varphi$, and so $a_{2d,0}\ne 0$.
Since $\beta _{2d,0}(k)=0$, one sees that $x^{4d} \in \varphi_k$ for
all $k$. If $k$ is odd, we see that $u^2z, x^{4d} \in
\varphi_{k,\sigma_{\rm odd}-wt=2d}$. Hence  $\varphi_{k,\sigma_{\rm
odd}-wt=2d}$ is irreducible. If  $k$ is even, we see that $y^2z,
x^{4d} \in \varphi_{k,\sigma_{\rm ev}-wt=2d}$. Hence
$\varphi_{k,\sigma_{\rm ev}-wt=2d}$ is irreducible. Since
$wt(\varphi)=2d$, we thus have
$$K_{Z_{k+1}}=\phi_k^*K_{Z_k}+(wt(xyzu)-wt(\varphi_k)-1)E_{k+1}=
\phi_k^*K_{Z_k}+\frac{1}{2}E_{k+1}.$$ This completes the proof.
\end{proof}

The other case with discrepancy $\frac{a}{2}$ can be treated similarly. We
may assume that $r+2=(2d+1)a$ for some $d>0$. The local equation of $P \in X$ is given by
$$\left\{ \begin{array}{ll}
\phi_1:&=u^2+yw+p(x^2,z)=0\\
       &=u^2+yw+\sum a_{ij}x^{2i}z^j=0\\
\phi_2:&=yz+x^{2d+1}+xzq(x^2,z)+w=0\\
       &=yz+x^{2d+1}+\sum b_{ij}x^{2i+1}z^{j+1}+w=0\\
\end{array} \right.
$$

The divisorial contraction $Y \to X$ with discrepancy $\frac{a}{2}$ is a
weighted blowup with weight $\sigma_a=\frac{1}{2}(a,r,2,r+2,r+4)$ with
respect to the coordinates $(x,y,z,u,w)$.
Thus the local equation for $Y$ near $Q_3$ is
$$\left\{ \begin{array}{ll}
\phi_{1,a}&=u^2+yw+\sum a_{ij}x^{2i}z^{ai+j-(r+2)}=0\\
\phi_{2,a}&=y+x^{2d+1}+\sum b_{ij}x^{2i+1}z^{j+1+\frac{(2i+1)a-(r+2)}{2}}+wz=0\\
\end{array} \right.
$$

Similarly to the previous case, we construct a sequence of weighted blow-ups $Z_a \to Z_{a-1} \to \ldots \to Z_1 \to Z_0=X$
by weights $\sigma_{\rm ev}=\frac{1}{2}(1,2d-1,2,2d+1,2d+3)$ and $\sigma_{\rm odd}=\frac{1}{2}(1,2d+3,2,2d+1,2d-1)$,
where $Z_{k+1} \to Z_k$ is a weighted blowup with weight $\sigma_{\rm ev}$ (resp. $\sigma_{\rm odd}$) if $k$ is even (resp. odd).

By a direct computation, one sees that
if $k$ is odd, then the local equations are
$$\left\{ \begin{array}{ll}
\phi_{1,k}&=u^2+yw+\sum a_{ij}x^{2i}z^{ki+j-k(2d+1)}=0\\
\phi_{2,k}&=y+x^{2d+1}+\sum b_{ij}x^{2i+1}z^{j+1+\frac{k(2i+1)-k(2d+1)}{2}}+wz=0\\
\end{array} \right.
$$

If $k$ is even, then the local equations are
$$\left\{ \begin{array}{ll}
\phi_{1,k}&=u^2+yw+\sum a_{ij}x^{2i}z^{ki+j-k(2d+1)}=0\\
\phi_{2,k}&=yz+x^{2d+1}+\sum b_{ij}x^{2i+1}z^{j+1+\frac{k(2i+1)-k(2d+1)}{2}}+w=0\\
\end{array} \right.
$$

Notice that  we have $wt(\phi_1)=2d+1$  (resp. $wt(\phi_2)=\frac{2d+1}{2}$) because
$ki+j-k(2d+1) \ge 0$ (resp. $j+1+\frac{k(2i+1)-k(2d+1)}{2} \ge 0$) for all $1\leq k\leq a$ and all $i$ and $j$ such that $a_{i,j}\ne 0$ (resp. $b_{i,j}\ne 0$).

We now compare the singularities. On $Y$, there are cyclic
quotient points $Q_2$ and  $Q_5$ of index $r$ and $r+4$
respectively and a $cD/2$ point $Q_3$. On $Z_k$, there are two
cyclic quotient point $R^+_k, R^-_k$ of index $2d+3, 2d-1$
respectively and a $cD/2$ point $R^3_k$. The point $R^3_a\in Z_a$
has index $2$ and is isomorphic to  $Q_3$ in $Y$. Hence we
conclude that
$$\dep(X) \le a+\dep(R^3_a)+\sum_{k=1}^a (\dep(R^+_k)+\dep(R^-_k))$$
$$=a+\dep(Q_3)+a(2d+2+2d-2)=a+\dep(Q_3)+4ad$$
$$=\dep(Q_3)+2r+4-a=\dep(Y)+2-a.$$
Since $a \ge 3$, we have $\dep(Y) \ge \dep(X)+a-2 \ge \dep(X)+1$.
\end{proof}

\section{Decomposition}
The purpose of this section is to prove the main result \eqref{divSig}.
We begin by proving the following key result.

\begin{thm} \label{key} Let $g:X\supset C\to W\ni R$ be an extremal neighborhood.
If $X$ is not Gorenstein, then there exists a $w$-morphism $f: Y \to X$,
such that $C_Y \cdot K_Y \le 0$, where $C_Y$ denotes the proper
transform of $C$ in $Y$.
\end{thm}

\begin{proof}
We use the classification of extremal neighborhoods. Each case is
elementary.

Let $P\in C\subset X$ be an isolated terminal singularity of index $r>1$.
By \eqref{Ha}, there is a weighted blowup $f: Y \to
X$ with exceptional divisor $F$
such that $K_Y=f^*K_X+\frac{1}{r} F$. We will need to compute
\begin{equation}\label{e-1}C_Y \cdot K_Y= C_Y \cdot (f^*K_X+ \frac{1}{r}F)=C\cdot K_X+ \frac{1}{r} C_Y \cdot F.\end{equation}

We follow the classification summarized in Appendix A and B of
\cite{Mo88} together with the classification of the terminal singular point
$P$ cf. \cite[pg. 541]{KM92}, and we will apply the results of \cite{HaI} and \cite{HaII}.

Suppose that $C\not\subset E_X$ where $E_X$ is a general section of $|-K_X|$.
By \eqref{t-ge}, $E_X$ is normal with Du Val singularities.
By \eqref{l-e2}, we have that
$f^*(K_X+E_X)=K_Y+E_Y$ where $E_Y=f^{-1}_*E_X$.
Therefore $$0=C\cdot (K_X+E_X)=C_Y\cdot (K_Y+E_Y)\geq C_Y\cdot K_Y.$$
It suffices therefore to consider those extremal neighborhoods with $C\subset E_X$. By \cite[2.2]{KM92} and \cite{Mo07}, the only cases to consider are IC, IIB, exceptional IA + IA, IA + IA + III and semistable IA + IA.

{\bf IC.} By \cite[pg. 243]{Mo88}, we have that $$C^\sharp = \{y_1^{r-2}-y_2^2=y_4=0\}\subset $$
$$\qquad \qquad X^\sharp =\{y_1,y_2,y_4 \}/\mathbb Z _r (2,r-2,1)$$
where $r\geq 5$ is odd.
Following \cite[3.9]{HaI}, if we blow up $\mathbb C ^3/\mathbb Z _r (2,r-2,1)$ with weight $\frac 1 r (2,r-2,1)$, we obtain a divisorial extraction with discrepancy $1/r$.

Computing on $U_1$ yields $C_Y = \{1-\bar y _2^2=\bar y _4=0 \}$ and $F=\{\bar y _1=0 \}$
where $\mathbb Z _2$ acts with weights $(r,-r+2,-1)$. Since $r$ is odd,
on $U_1$ we have
$F\cap C_Y=Q$
is a point with multiplicity $1$.

Computing on $U_2$ yields $C _Y = \{\bar y _1^{r-2}-1=\bar y _4=0 \}$ and $F=\{\bar y _2=0 \}$
where $\mathbb Z _{r-2}$ acts with weights $(-2,r,-1)$.
So  on $U_2$ we have
$F\cap C_Y=Q$ where $Q\in U_1\cap U_2$ is the same point as above.

Computing on $U_3$ yields no further intersection points.
Therefore, $C_Y\cdot F=1$ and hence $K_Y\cdot C_Y\leq 0$.

{\bf IIB.} We have $cAx/4$ point $P \in C$.
By \cite[pg. 243]{Mo88}, we have $$C^\sharp = \{y_1^{2}-y_2^3=y_3=y_4=0\}\subset $$
$$\qquad X^\sharp =\{ y_1^{2}-y_2^3+g_3y_3+g_4y_4=0 \}$$ where
$\mathbb Z _4$ acts on $\mathbb C^4$
with weights
$(3,2,1,1)$. By \cite[\S 7]{HaI} there is a $w$-morphism $f:Y\supset F\to X\ni P$ which
is a weighted blowup
 with equivalent weights $\frac{1}{4}(r_1,r_2,r_3,r_4)$ such that $r_1 \equiv 3, r_2 \equiv 2, r_3 \equiv r_4 \equiv 1 \ \  (\text{mod} \ 4)$.

We distinguish the following cases:\\
 If $2r_1 > 3r_2$, then computation on $U_1$ yields that
$C_Y =\{
\bar y_1^{\frac{2r_1-3r_2}{4}}-\bar{y}_2^{3}=\bar{y}_3=\bar{y}_4=0\}$,
$F$ is defined by $\bar{y}_1=0$ where $\mathbb Z _{r_1}$ acts with weights $(4,-r_2,-r_3,-r_4)$.
Moreover, computation on $U_2$, $U_3$ and $U_4$ yields no further points in $C_Y \cap F $.
Thus $C_Y \cdot F = \frac{3}{r_1}  \le 1$.

If $2r_1<3r_2$, then similar computation yields that $C_Y \cdot F = \frac{2}{r_2}  \le 1$.

Finally, if $2r_1=3r_2$, then a similar computation
that $C_Y \cdot F = \frac{3}{r_1}=\frac{2}{r_2}$.

In conclusion we have that
$$C_Y \cdot F = \min\{ \frac{3}{r_1},\frac{2}{r_2}\} \le 1, $$
and so $C_Y\cdot K_Y\leq 0.$

Before we turn our attention to the remaining three cases, it is necessary to prove the following.
\begin{lem}
  [Blowing up $cA/r$ points on an extremal neighborhood of type $IA$]\label{l-buIA} Let $P\in C\subset X$ be a point of type $cA/r$ on an extremal neighborhood of type $IA$, then by \cite[pg. 243]{Mo88},
we have $$C^\sharp =\{ y_1^{a_2}-y_2^{a_1}=y_3=y_4=0\} \subset$$ $$\qquad X^\sharp =\{ \varphi= g_2
(y_1^{a_2}-y_2^{a_1})+g_3y_3+g_4y_4=0\}$$ where $\mathbb Z_r $ acts on
$\mathbb  C^4$ with weights $(a_1,a_2,-a_1,0)$ and $(a_1a_2,r)=1$.

Then there is a $w$-morphism $Y\supset F\to X\ni P$
of discrepancy $1/r$ given by a weight $\sigma = \frac 1 r (r_1,1,r_2,r)$
such that $C_Y\cdot F =a_1/r_1$ where $C_Y$ is the strict transform of $C$
and $r_1\equiv -r_2\equiv a_ 1a_2^{-1}$ modulo $r$.
\end{lem}
\begin{proof}
Let  $r>\lambda>0$ be an integer such that $a_2 \lambda =qr+1$ for some integer $q\geq 0$
and $\alpha$ be the residue of $a_1 \lambda$ modulo $r$.
By \cite[Lemma 6.5]{HaI},
there is a $w$-morphism $f:Y\supset F\to X \ni P$ of discrepancy $1/r$ given by a weight $\sigma = \frac 1 r (r_1,1,r_2,r)$ where
$r_1$ and $r_2$ are positive integers such that $r_1\equiv -r_2\equiv \alpha$ modulo $r$.

The weighted blow up is
covered by four affine pieces $U_1,...,U_4$ whose equations are described in \cite[\S 6]{HaI}. An easy computation shows that $F \cap C_Y = \emptyset $ on $U_3$ and
$U_4$.
 Since $r_1
\equiv a_1 \lambda \  (\text{mod} \  r)$, we may write $r_1=a_1 \lambda+tr$ for some $t$ . Hence
\begin{equation}\label{e-3} a_2r_1=a_1a_2 \lambda+a_2tr=(a_1q+a_2t)r+a_1.\end{equation}
Note that as $a_2r_1 >0$ and $r> a_1>0$, we have that $a_1q+a_2t \ge 0$.
Computation in $U_2$ now yields that
$C_Y =\{\bar{y}_1^{a_2}\bar{y}_2^{a_1q+a_2t}-1=\bar{y}_3=\bar{y}_4=0\}$,
$F =\{\bar{y}_2=0\}$.
Therefore $C_Y \cap F = \emptyset$ in $U_2$.

Computation in $U_1$ yields that
$C_Y =\{ \bar{y}_1^{a_1q+a_2t}-\bar{y}_2^{a_1}=\bar{y}_3=\bar{y}_4=0\}$,
$F =\{\bar{y}_1=0\}$ where $\mathbb Z _{r_1}$ acts with weights $(r,-1,-r_2,-r)$.
Therefore $C_Y \cdot F = a_1/r_1$.
\end{proof}

{\bf Exceptional IA + IA.} If $g$ is a flipping contraction (resp. a
divisorial contraction), then we have two IA points $P$ and $Q$ such that
$Q$ is a $cA/2$ (resp. $cA/2$ or $cAx/2$ or $cD/2$) point  and
$P$ is a
$cA/r$ point of odd index $r\ge 5$ (resp. $r\geq 3$).
Since (cf. \S \ref{ss-en})  $$C\cdot K_X=-1+\frac{1}{2}+\frac{r-a_2}{r}= \frac{1}{2}-\frac{a_2}{r}<0$$ we have $a_2
> \frac{r}{2}$.  Let $s=2a_2-r>0$. Then $C\cdot K_X={-s}/{2r}$.

By \cite[2.13.1]{KM92}, we have $a_1=1$. We consider the $w$-morphism $f:Y\supset F\to X\ni P$ over
the point $P$ given by \eqref{l-buIA}. We thus have $C_Y \cdot F =
\frac{1}{r_1}$ and hence $$K_Y\cdot C_Y=\frac {-s}{2r}+ \frac 1 {rr_1}=\frac 1 r (-\frac s 2+\frac 1 {r_1}).$$

Notice that $a_2r_1\equiv a_1=1$ modulo $r$ and so
$a_2r_1=1+qr$ for some integer $q\geq 0$ and hence $sr_1=2+(2q-r_1) r$. Since $sr_1> 0$ and $r>1$,
it follows that $sr_1 \ge 2$ and hence
${1}/{r_1} - {s}/{2}\le 0$. Therefore, we conclude that $C_Y \cdot K_Y \le 0$.


{\bf Semistable IA + IA.} Let $C$ be the flipping curve. Then $C$ contains two IA points $P$ and $P'$ corresponding to singularities of type $cA/r$ and $cA/r'$ respectively. We assume that $r\geq r'$.
By \cite[\S 13, pg. 679]{KM92},
the curve $C^\sharp$ is given by the $x$-axis, and $X^\sharp$ is given by $xy-z^{dr}+tf(x,y,z,t)=0\subset \mathbb C ^4$ with weights $\frac{1}{r}(1,-1,a,0)$.
It follows that $a_1=1$ and $a_2=a$.
Similarly for the second singular point we have $a'_1=1$ and $a'_2=a'$.

Let $\frac{a}{r}+\frac{a'}{r'}=1+\frac{\delta}{rr'}$, then $\delta >0$ and $K_X \cdot C=-\frac{\delta}{rr'}$. We write
\begin{equation}\label{e-s1} ar'+a'r=rr'+\delta.\end{equation}
Let $f:Y\supset F\to X \ni P$ be the weighted
blow up of $X$ at $P$ with weight $\sigma = \frac 1 r (r_1,1,r_2,r)$ as described in \eqref{l-buIA}. In particular $C_Y\cdot F=1/r_1$.
Recall that since $r_1\equiv a^{-1}$ $({\rm mod}\ r)$, we have that
\begin{equation}\label{e-s2} ar_1=\gamma r+1\ {\rm for\ some}\ \gamma \geq 0.\end{equation}

Suppose that $\gamma =0$, then $ar_1=1$. Hence $a=r_1=1$. Then $$\frac{r'-1}{r'} \ge \frac{a'}{r'} =1+\frac {\delta }{rr'}-\frac 1 r> \frac{r-1}{r},$$ so that $r'>r$ which is a contradiction.

Therefore, we have $\gamma >0$.  By \eqref{e-s2} and \eqref{e-s1}, we have
$$ ar_1 \delta \equiv \delta  \equiv ar' \ ({\rm mod}\ r).$$ Hence $ r_1 \delta \equiv r'\ ({\rm mod}\ r)$. Since $r' \leq r$,
we thus have $r_1 \delta  \ge r'$. It follows that $$C_Y \cdot K_Y=-\frac{\delta}{rr'}+ \frac{1}{rr_1} \le -\frac{\delta}{rr'}+ \frac{\delta}{rr'}=0.$$

{\bf IA + IA + III.} In this case $C\subset X$ is divisorial and the two IA points are both ordinary of $cA$ type of index $2$ and $r\geq 3$.
By \cite[2.12.1]{KM92}, we also have that $a_1=1$ at the $cA/r$ point.
The same computation as in the Exceptional IA+IA case
shows that $C_Y\cdot K_Y\leq 0$.
\end{proof}

\begin{thm}\label{decomp} Let $g: X\supset C \to W\ni P$ be an extremal neighborhood which is isolated (resp.
divisorial). If $X$ is not Gorenstein, then we have a diagram
\begin{diagram}
Y     &   & \rDashto &    &  Y '  \\
\dTo^{f} &       &      &   &  \dTo_{f'}      \\
 X       &        &      &  &  X'  \\
           & \rdTo^{g}  &      & \ldTo^{g'}  &        \\
           &        &  W   &        &
\end{diagram}
where $ Y \dashrightarrow Y'$ consists of flips and flops over $W$, $f$ is
a $w$-morphism, $f'$ is a divisorial contraction (resp.
a divisorial contraction
to a curve) and $g':X'\cong X^+\to W$ is the flip of $g$ (resp.
$g'$ is divisorial contraction to a point).
\end{thm}

\begin{proof}

{\bf  1. } By \cite[Lemma 3.15]{Mo82}, we have that any $g$-exceptional curve
$C'$ is not
numerically effective. In particular, there is  an effective
irreducible divisor $D \in {\rm Div}(X)$ such that $D \cdot C' <0$.

In the case of a divisorial contraction, this divisor $D$ is nothing
but the exceptional divisor to be contracted. In the case of small
contraction, the divisor $D$ moves in a family.

{\bf  2.} Let $f:Y \to X$ be the $w$-morphism  with  exceptional
divisor $F$ given by \eqref{key}. Let $l$ be a curve on $F$.  Let $C_Y$ and $D_Y$ be the
proper transforms of $C$ and $D$ respectively.

Since $\rho(Y/W)=2$ and ${\rm Exc}(g\circ f)= F\cup C_Y$,
$${\rm NE}(Y/W)= \bR_+ l + \bR_+ C_Y.$$

{\bf  3.} One can flip or flop along $C_Y$.\\
By $2$, $C_Y$ is extremal. Moreover, $f^*D=D_Y+\lambda F$ where $\lambda >0$
and hence
$$C_Y \cdot (K_Y+\epsilon D_Y) \le C_Y \cdot \epsilon D_Y=\epsilon
C_Y \cdot (f^*D-\lambda F) <0.$$

 Thus, we have a
$(K_Y+\epsilon D_Y)$-flip $Y \dashrightarrow Y^+$, which is a $K_Y$-flip
or $K_Y$-flop.

{\bf 4.} Notice that $K_{Y^+}$ is not nef over $W$.\\ More
precisely, let $l \subset F$ be a general curve, then $l \cap
C_Y=\emptyset$. Hence the proper transform  $l_+$ of $l$ in $Y^+$
satisfies $l_+ \cdot K_{Y^+}=l \cdot K_Y <0$. Moreover, the family of such curves dominates the divisor $F_{Y^+}$.

We now run the $K_{Y^+}$ minimal model program over $W$.
By termination of flips,
we must end up with a divisorial contraction $Y' \to X'$.
Therefore,  we have the required diagram.

{\bf  5.} Suppose that $g: X \to W$ is divisorial with
exceptional divisor $D$.
Notice that $${\rm Exc}(Y'/W) \supset D_{Y'}
\cup F_{Y'}$$ where $D_{Y'}$ and $F_{Y'}$ are the strict transforms of $F$ and $D$ on $Y$.

We claim that $f'$  contracts the divisor $D_{Y'}$
to a curve.

Suppose to the contrary that $Y' \to X' $ contracts
$F_{Y'}$. Then we see that $X' \to W$ has relative
Picard number $\rho(X'/W)=1$ containing an exceptional divisor
$D_{X'}$. Notice furthermore that $D$ and $D_{X'}$
are isomorphic away from finitely many curves. Therefore, we have
$X' \cong X$ by an easy argument cf.
\cite[Lemma 3.4]{Kk01}. By the same argument, we also have $Y
\cong Y'$, which is a contradiction.

Hence $f': Y' \to X'$ contracts $D_{Y'}$.
Notice that $g(D)$ is a curve in $W$. Hence $g'(
f'(D_{Y'}))$ is a curve. In particular,
$f'(D_{Y'})$ is a curve.

Since $g(f(F))$ is a point. It follows that
$g'(F_{X'})$ is a point.

{\bf 6.} Suppose  that $g: X \to W$ is isolated. We claim that
$X' \cong X^+$ is the flip.

First of all, $X$ and $X'$ are isomorphic  in codimension $1$ and the relative Picard number $\rho(X'/W)$ is equal to $1$.

If $K_{X'}$ is $g'$-ample over $W$, we are
done.

If $K_{X'}$ is $g'$-trivial, we have that $K_{X'}=g'^*K_W$. This implies that  $K_W$ is $\mathbb{Q}$-Cartier, which is impossible.

If $K_{X'}$ is $g'$-negative, we see that $$ X \cong \textrm{Proj} _W\left( \bigoplus_{m\ge 0}\mathcal  O_W(-mK_W)\right) \cong X'.$$
It is clear that  $F_Y$ and $F_{Y'}$ correspond to the same valuation in the function field. We thus have
$Y \cong Y'$ by  \cite[Lemma 3.4]{Kk01}, which is absurd because $Y \dasharrow Y'$ is a composition of a sequence of $(K_Y+\epsilon D_Y)$-flips.
\end{proof}

\begin{rem} \label{fl}
We would like to remark that if $g: X \to W$ is isolated, then
either $f': Y' \to X'$ is a divisorial contraction to a curve or
at least one of $Y_i \dasharrow Y_{i+1}$ is a flip. To see this, pick any
flipped curve $C_{X'}$ such that $K_{X'} \cdot C_{X'} >0$. Suppose
that $f'$ is a divisorial contraction to a point, then $C_{X'}\not\subset f'(F_{Y'})$. Hence $$K_{Y'}  \cdot {C_{Y'}}=(f'^*K_{X'}+a
F_{Y'})\cdot {C_{Y'}}= K_{X'} \cdot C_{X'}+ a{C_{Y'}} \cdot
F_{Y'}
>0,$$ where ${C_{Y'}}$ denotes the proper transform of $C_{X'}$
in $Y'$ and $a >0$. If there are no flips $Y_i
 \dasharrow  Y_{i+1}$, then $Y\dasharrow Y'$ is a flop and if $C_Y$ is the
 proper transform of $C_{Y'}$ in $Y$, then $K_Y\cdot C_Y=K_{Y'}\cdot C_{Y'}>0$.
Therefore, as $f:Y\to X$ is a divisorial contraction, $C_Y$ maps to a curve
$C_X\subset X$ which is the flipping curve. But 
then $C_Y$ is (by construction) the flopping curve of $Y\dasharrow Y'$ and so $C_Y\cdot K_Y=0$. This is the required contradiction.
\end{rem}

\begin{rem}\label{r-flip} The result of \eqref{decomp} also holds if
we assume that $g:X\supset C \to W\ni P$ is a flipping contraction
or a divisorial contraction to a curve.
In fact, by \eqref{t-flip} and \eqref{r-div}, $g$ factors through an extremal neighborhood $g':X\supset C '\to W'\ni P'$. Therefore, \eqref{key} also applies in this context so that we have a $w$-morphism $f:Y\to X$ such that $K_Y\cdot C_Y\leq 0$.
The rest of the argument is the same as the proof of \eqref{t-flip}.
\end{rem}

\begin{rem}
In order to prove \eqref{t-ge}, we will need an invariant which is minimal
for Gorenstein varieties and that improves
under $w$-morphisms and flips. There seem to be two possible approaches,
the first one using general elephants $E_X\in |-K_X|$ (i.e. assuming that there exists an element $E_X\in |-K_X|$ with Du Val singularities) and the second one using $\dep(X)$.

It is expected that if $g:X\to W$ is the germ of a flipping or divisorial contraction, then the general elements $E_X\in |-K_X|$ and $E_W\in |-K_W|$ have Du Val singularities (this is known in the divisorial to point case by \cite{Kk05}, in the extremal neighborhood case by \eqref{t-ge} and is expected to hold unconditionally).
If we assume that $E_X$ is Du Val, then there is a simpler proof for \eqref{divSig} not involving the invariants $\dep(X)$ which we now explain.
\end{rem}
\begin{proof}[Proof of \eqref{divSig} assuming that $E_X$ is Du Val.]
We may work locally over $W$ and so we may assume that
there is a point $Q\in W$ such that $W-Q$ is smooth and $X-f^{-1}(Q)$ is smooth.
Let $E_X$ be as above and let $\tilde E_{X}$ be the minimal resolution of $E_{X}$ then $\rho (\tilde E_{X}/E_{X})$ defines an invariant of (the germ of) $g:X\to W$.
We will proceed by induction on $\rho (\tilde E_{X}/E_{X})$.

If $\rho (\tilde E_{X}/E_{X})=0$, then $E_{X}$ is smooth
and hence terminal. By inversion of adjunction $(X,E_{X})$ is terminal.
By classification of terminal $3$-fold singularities, $X$ is Gorenstein. By \eqref{t-Cu},
$g$ is divisorial and it is the blow up of a LCI curve
in $W$.

We may therefore assume that $\rho(\tilde E _X/E_X)>0$.
By \eqref{r-flip}, we may factor $X\to W$ by rational maps over $W$
$$X\stackrel{f}{\leftarrow} Y \dasharrow Y' \stackrel{f'}{\to} X'.$$

Let $f:F\subset Y\to P\in X$ be the $f$-exceptional divisor.
By \eqref{l-e2}, $f^*(K_{X}+E_{X})=K_{Y}+E_{Y}$ where  $E_{Y}=f^{-1}_*E_X$.
By \eqref{l-e2}, $\tilde E_{X}\to E_{X}$ factors through $E_{Y}$.
It follows by \eqref{l-e1} that $\rho (\tilde E_Y/E_Y)<\rho (\tilde E_X/E_X)$.
The rational map
$Y\dasharrow Y'$ is given by a sequence of flips and flops $Y_i\dasharrow Y_{i+1}$.
\begin{claim} If $Y_i\dasharrow Y_{i+1}$ is a flip (resp. a flop), then
$\rho (\tilde E_{Y_i}/E_{Y_i})\geq\rho (\tilde E_{Y_{i+1}}/E_{Y_{i+1}})$ where $E_{Y_i}$ denotes the strict transform of $E_X$ and $\tilde E_{Y_i}$ denotes its minimal resolution.\end{claim}
\begin{proof}[Proof of the Claim.]
By \eqref{t-flip}, we may assume that the flip $Y_i\dasharrow Y_{i+1}$
factors in the analytic category, through a sequence of flips
with irreducible flipping curve.
Let $\bar Y _i\to \bar W_i\leftarrow \bar Y _i^+$ be one
of these flips. It suffices to show that $\rho (\tilde E_{\bar Y _i}/E_{\bar Y _i})\geq
\rho (\tilde E_{\bar Y ^+_i}/E_{\bar Y ^+_i})$.
Since $\bar Y _i\dasharrow \bar Y _i^+$ is an isomorphism on the complement
of the flipping curve $\bar C _i$, we may assume that $\bar Y _i\supset
\bar C _i\to \bar W_i\ni \bar w _i$ is an extremal neighborhood.
Since $\bar Y _i\dasharrow \bar Y _i^+$ is a $(K_{\bar Y _i}+E_{\bar Y _i})$-flop,
we have that ${E_{\bar Y _i^+}\to E_{\bar W _i}}$ is crepant i.e.
$K_{E_{\bar Y _i^+}/E_{\bar W _i}}=0$.
Therefore, if $E_{\bar Y _i}$ has
Du Val singularities, then $E_{\bar Y ^+_i}\in |-K_{\bar Y ^+_i}|$ is also Du Val.
Moreover, the rational map $\tilde E_{\bar Y _i}\dasharrow E_{\bar Y _i^+}$ is
a morphism. Since the flipping curve is irreducible,
$\rho(E_{\bar Y _i}/E_{\bar W _i})\leq 1$. Since $E_{\bar Y _i^+}\in |-K_{\bar Y _i^+}|$ and
$K_{\bar Y _i^+}$ is ample over $\bar W _i$, $E_{\bar Y _i^+}$ contains the flipped
curves so that $\rho(E_{\bar Y _i^+}/E_{\bar W _i})\geq 1$.
Therefore, we have $$\rho(\tilde E_{\bar Y ^+_i}/E_{\bar Y _i^+})\leq \rho(\tilde E_{\bar Y _i}/E_{\bar Y _i^+})\leq \rho(\tilde E_{\bar Y _i}/E_{\bar Y _i}).$$

The case of a flop is similar. In this case $K_{\bar Y _i}\equiv _{\bar W _i}0$
and $K_{\bar Y _i^+}\equiv _{\bar W _i}0$. If $E_{\bar Y _i}$ does not intersect the
flopping locus then $E_{\bar Y ^+_i}$ does not intersect the
flopping locus. If $E_{\bar Y _i}$ intersects the
flopping locus, then $E_{\bar Y ^+_i}$  intersects the flopped locus and hence
it contains the flopped locus. Therefore  $\rho(\tilde E_{\bar Y _i}/E_{\bar Y _i^+})\leq \rho(\tilde E_{\bar Y _i}/E_{\bar Y _i})$ and we conclude as in the previous case.
\end{proof}
It follows that if $G_{Y_i}\in |-K_{Y_i}|$ is general, then it has Du Val singularities and
$$\rho  (\tilde E_X/E_X)>\rho  (\tilde E_{Y_i}/E_{Y_i})\geq \rho  (\tilde G_{Y_i}/G_{Y_i})\qquad {\rm for\ all}\  i.$$
By induction, we obtain the required factorization of $Y\dasharrow Y'$ and of
$Y'\to X'$. $\qed$
\end{proof}
We will now prove \eqref{divSig} in full generality using $\dep
(X)$. We will need the following propositions that we believe are of
independent interest.
\begin{prop}\label{p-fl} Let $X\to W$ be a flipping contraction and let $X\dasharrow X'$ be the flip. Then $\dep (X) >  \dep (X')$.
\end{prop}
\begin{prop}\label{p-div} Let $X\to W$ be a extremal contraction to a curve. Then $\dep (X)\geq \dep (W)$.
\end{prop}

\begin{proof}[Proof of \eqref{divSig} and \eqref{p-fl}, \eqref{p-div}.]
We may work locally over the base $W$ and so we may assume that $g:X\to W$ is the germ of a flipping (resp. a divisorial to curve) contraction.
By \eqref{r-flip}, we may factor $X\dasharrow X'$ (resp. $X\to W$) by rational maps over $W$
$$X\stackrel{f}{\leftarrow} Y \dasharrow Y' \stackrel{f'}{\to} X'.$$
We proceed by induction on $\dep ( X)$. If $\dep (X)=0$, then $X$
is Gorenstein. There is no flipping contraction. Therefore $g$ is
a divisorial contraction to a curve, given by blowing up a LCI
curve in $W$. Hence the above two Propositions and Theorem
\ref{divSig} holds when $\dep(X)=0$.

We may therefore assume that $\dep ( X)>0$ and the statement holds
for all threefolds with $\dep < \dep(X)$.

By the proof of \eqref{key}, we may assume that $Y \to X$ is the first weighted blowup of a
$w$-resolution of
$X$, so that  $\dep(X)-1=\dep (Y)$. If $Y\dasharrow Y^+$ is a
flop, then by \cite{Kol89} $Y$ and $Y^+$ have isomorphic
singularities and so $\dep (Y)=\dep (Y^+)$.

Suppose first that $f: X \to W$ is a flipping contraction. By
induction, we then have $\dep (Y)\geq \ldots \geq \dep (Y_i) \geq
\ldots \geq \dep (Y')$. By induction on $\dep (\ldots )$ we obtain
the required factorization for each $Y_i\dasharrow Y_{i+1}$ and
hence for $X\dasharrow X'$.

Indeed, by Remark \ref{fl}, either $Y' \to X'$ is a divisorial
contraction to a curve, or at least one of $Y_i \to Y_{i+1}$ is a
flip.

If $Y' \to X'$ is a divisorial contraction to a curve, then by
induction, $\dep (Y')\geq \dep (X')$. Therefore, $$\dep(X)
=\dep(Y)+1 > \dep(Y) \ge \dep(Y')  \geq \dep(X').$$ If $Y' \to X'$
is a divisorial contraction to a point, then by Remark \ref{fl},
we have that at least one of $Y_i \to Y_{i+1}$ is a flip.
Therefore, $\dep(Y_i) > \dep(Y_{i+1})$ by induction and hence
$\dep(Y) > \dep(Y')$. Together with Proposition \ref{p-dep} that
$\dep(Y')>\dep(X)-1$, it follows that
$$\dep(X)
=\dep(Y)+1 > \dep(Y')+1  \geq \dep(X').$$

Suppose next that $f: X \to W$ is a divisorial contraction to a
curve. Again, we have $\dep (Y)\geq \ldots \geq \dep (Y_i) \geq
\ldots \geq \dep (Y')$. Since $Y' \to X'$ is a divisorial
contraction to a curve and $\dep(Y') < \dep(X)$, we have $\dep(Y')
\ge \dep(X')$. Finally, as $X' \to W$ is a divisorial contraction to
a point, we have that $\dep(X') \ge \dep(W)-1$. We thus conclude
that $$ \dep(X)=\dep(Y)+1 \ge \dep(X')+1 \ge \dep(W).$$ Moreover,
by induction on $\dep (\ldots )$ we obtain the required factorization for
each $Y_i\dasharrow Y_{i+1}$ and hence for $X \dasharrow W$.
\end{proof}


\end{document}